\documentclass[a4paper,11pt]{amsart}

\usepackage{hyperref,latexsym}

\usepackage{amsfonts, epsfig, amsmath, amssymb, color, amscd}

\usepackage{enumerate}
\usepackage{graphicx}

\theoremstyle{plain}
\newtheorem{theorem}{Theorem}[section]

\newtheorem{corollary}[theorem]{Corollary}
\newtheorem{proposition}[theorem]{Proposition}
\theoremstyle{definition}
\newtheorem{definition}[theorem]{Definition}
\newtheorem{example}{Example}
\theoremstyle{remark}
\newtheorem{remark}{Remark}

\newcommand{\M}{\mathcal{M}}

\begin{document}

\title[ Billiard problems]
      {A survey on Polynomial in momenta integrals for billiard problems}

\date{March 2018}
\author{Misha Bialy and Andrey E. Mironov}
\address{M. Bialy, School of Mathematical Sciences, Tel Aviv
University, Israel} \email{bialy@post.tau.ac.il}
\address{A.E. Mironov,
 Sobolev Institute of Mathematics, 4 Acad. Koptyug avenue,
and Novosibirsk State University, Pirogova st 1,
630090, Novosibirsk, Russia}
\email{mironov@math.nsc.ru}
\thanks{M.B. was supported in part by the Israel Science Foundation grant
162/15}
\thanks{ A.M. was supported by the Laboratory of Topology and Dynamics, Novosibirsk State University (contract no. 14.Y26.31.0025 with the Ministry of Education and Science of the Russian Federation).}

\subjclass[2010]{37J40,37J35} \keywords{{
polynomial integrals, Birkhoff conjecture, magnetic billiards, Two-sided magnetic billiards}}

\begin{abstract}
	In this paper we give a short survey of recent results on algebraic version of the Birkhoff
	conjecture for integrable billiards on surfaces of constant curvature. We also discuss integrable
	magnetic billiards. As a new application of the algebraic technique we  study the existence of polynomial integrals for the  two-sided
	magnetic billiards introduced by Kozlov and Polikarpov.
	
\end{abstract}

\maketitle



\section{\bf Introduction}
In this survey paper we describe an algebraic approach to a very old conjecture attributed to G. Birkhoff on billiard dynamics.
The conjecture was explicitly formulated in \cite{poritsky} and since then remains unsolved in full generality.
The question is if the only integrable convex billiards in the plane are ellipses.

Here it is crucial to specify what is understood by integrability. Accordingly, there are various approaches to this conjecture.
In this paper we shall discuss only one possible  approach to the problem --- namely algebraic one. In this approach it is assumed that there exists a first integral of the billiard ball motion which is polynomial in velocities. It is very natural to consider this class of integrals from the point of view of classical mechanics.

We are not going to discuss here many other developments in the direction of Birkhoff conjecture, but only mention them.
In \cite {delshams} it is shown that perturbations of ellipses create splitting of separatrices.
In \cite{B1} (see also \cite{W}) it is proved that the only billiards with the phase cylinder foliated by rotational invariant curves are circles.
In \cite{Innami} (see also \cite{BA} for geometric approach) it is proved that if there exist sequence of convex caustics with the rotation numbers tending to one half, then the billiard is an ellipse. In \cite{treshsev} an evidence of possible integrable dynamics around a 2-periodic orbits is given. In \cite{Marco} polynomial entropy approach to the problem is suggested.
Finally we will not discuss here series of recent results by V. Kaloshin et al. (\cite{KS} with references therein)  proving a local version of Birkhoff conjecture in a neighborhood of ellipses in a suitable functional spaces.

Algebraic approach to Birkhoff conjecture was initiated by S.~Bolotin \cite{bolotin}, \cite{bolotin2}. He studied billiards on the plane and on constant curvature surfaces.

Next very influential step was done by S. Tabachnikov in \cite{Tab} for the so-called outer billiards. His approach in a sense is dual to that of S.~Bolotin. It is conjectured in \cite{Tab} that if there exists a polynomial function which is preserved by outer billiard dynamics, then the curve must be an ellipse.  Recently A.~Glutsyuk and E. Shustin \cite{g2} confirmed this conjecture in affirmative.

Next step for Birkhoff billiards was done in our paper \cite{BM1}.
In \cite{BM1} we introduced the {\it angular billiard} which is dual to the Birkhoff billiard. Using angular billiard we obtained new results on Birkhoff billiard (see bellow). In particular we derived and studied a remarkable equation similar to one studied in \cite{Tab}. Analogous results were obtained also for Birkhoff billiards on constant curvature surfaces \cite{BM2}.

Finally, A. Glutsyuk \cite{g1}, \cite{g3} using the results of \cite{BM1}, \cite{BM2} completed the proof of algebraic Birkhoff conjecture for billiards on the plane and constant curvature surfaces.

Let us remark that though algebraic approach is restricted to the class of algebraic curves and polynomial integrals for them, but on the other hand it does not require closeness to the ellipses. Moreover, this approach allows to consider piecewise smooth boundaries.

It turns out that the algebraic approach can be extended to the case of magnetic billiards. This extension is based on the interplay of differential and algebra-geometric properties of the equidistant curves of the boundary of billiard domain. We have implemented it in \cite{BM3} for the plane magnetic billiards and then in \cite{BM4} for magnetic billiards on the constant curvature surfaces. In the present paper we apply algebraic approach for the model of two-sided magnetic billiards, where the magnetic field changes sign for every reflection of the boundary. This model was introduced by
V. Kozlov and S. Polikarpov in \cite{k}.

The paper is organized as follows. In Sections 2, 3 we survey the results which lead to the solution of algebraic Birkhoff conjecture in the plane and constant curvature surfaces. In Section 4 we describe our results for magnetic billiards in the plane and on surfaces of constant curvature.
In the last Section 5 we treat the case of two-sided magnetic billiards.
This part is a new ingredient in the paper and therefore it is given with all the details.

\section {\bf Algebraic Birkhoff conjecture and Angular billiard}
Let $\Omega$ be a convex domain in ${\mathbf R}^2$ with the smooth boundary
$\gamma=\partial\Omega$. We consider the billiard motion of a particle in $\Omega$. The particle
moves along a straight line inside $\Omega$, while reaching the boundary $\gamma$
it is reflected according to the law of geometric optics. This dynamical system is called {\it Birkhoff billiard}. The Birkhoff billiard is called {\it algebraically integrable} if there is a polynomial in velocity $v=(v_1,v_2)$ first integral $F(x,y,v)$ which is a non-constant function on the energy level $\{|v|=1\}$. There are two
examples:

\vskip03mm

\begin{example} Let $\Omega$ be the interior of the circle
$$x^2+y^2=R^2,$$
 then Birkhoff billiard admits the first integral
 $$F=yv_1-xv_2.$$
\end{example}
\vskip03mm

\begin{example}
Let $\Omega$ be the interior of the ellipse
$$
 \frac{x^2}{a^2}+\frac{y^2}{b^2}=1,\quad a>b>0,
$$
then the Birkhoff billiard admits the first integral
$$F=b^2v_1^2+a^2v_2^2-(xv_2-yv_1)^2.$$
\end{example}
\vskip03mm

The algebraic version of the Birkhoff conjecture states that if a Birkhoff billiard is algebraically integrable, then $\gamma=\partial\Omega$ is an ellipse.

Recently this conjecture was completed by A. Glutsyuk \cite{g1} using our results on angular billiards \cite{BM1} which we will discuss below.

In \cite{bolotin} S. Bolotin  proved the following result.

\vskip4mm
\begin{theorem}( \cite{bolotin})  Assume
that Birkhoff billiard inside $\gamma$ admits a non-constant
polynomial integral $\Phi$ on the energy level $\{|v|=1\}.$ It then
follows that $\gamma$  is a real algebraic curve. Moreover, let $\tilde{\gamma}$
 be the corresponding irreducible curve in $\mathbf{C}P^2.$ Then, the
following alternative holds: either ${\gamma}$ is an ellipse, or $\tilde{\gamma}$ necessarily contains singular
points.
\end{theorem} \vskip3mm

Let us recall the construction of the angular billiard. Let $D\subset{\mathbf R}^2$ be a convex
domain with smooth boundary $\Gamma=\partial D$. We fix a point $O\in D$. For an arbitrary point
$A\in U={\mathbf R}^2\backslash D$ there are two tangent lines to $\Gamma$ passing through $A$. Let $l$ be the right tangent (if one looks at $\Gamma$ from $A$). There is a unique line $l_A$ passing through $O$ such that the angle between $l_A$ and $OT$ equals to the angle $AOT$ , where $T$ is the tangency point (see Fig. \ref{-1}).
\begin{figure}[h]
	\centering
	\includegraphics[width=0.5\linewidth]{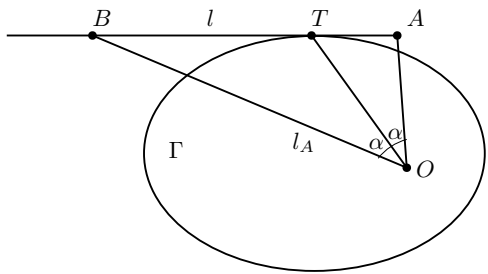}
	\caption{Equal angles $\angle AOT=\angle TOB$, $\mathcal A(A)=B$.}
	\label{-1}
\end{figure}
We get the mapping
$$
 {\mathcal A}:U\backslash S\rightarrow U,\qquad {\mathcal A}(A)=B,
$$
where $S=\{A:l_A\parallel AT\}$. The mapping ${\mathcal A}$ is called the {\it angular} billiard.
It turns out that the dynamics of the Birkhoff billiard for $\gamma$ is {\it equivalent} to the dynamics of the angular billiard for the polar dual curve
$\Gamma$. By definition, $\Gamma$ consists of the points which
are dual to the tangent lines of $\gamma.$

In order to explain this equivalence let us recall the geometric construction of polar duality correspondence between points and lines in ${\mathbf R}^2$. Traditionally, lines are denoted by small letters and corresponding
dual points are denoted by capital letters.
Fix a point $O\in{\mathbf R}^2$.
For a given line $l$ not passing
through $O$ denote by $p$ the distance from $O$ to $l$. Then the dual point $L$ corresponding to
$l$ is the
point lying on the normal radial ray to $l$ at the distance $1/p$ from $O$.

 Duality preserves the incidence relation and dual to
$\Gamma$ is $\gamma$ again. More precisely, if $t$ is tangent to
$\gamma$ at $L$ then the dual line $l$ is tangent to $\Gamma$ at $T$
(see Fig. \ref{01}).

Furthermore, suppose the particle moving along line $a$ after the collision with $\gamma$ at $L$ is reflected to the line $b$. It then follows that the dual points
$A,B$ lie on the line $l$ which is tangent to $\Gamma$ at $T$ (see Fig. \ref{01}).
Moreover it is easy to see that the angles $AOT$ and $BOT$ are equal, so the rule of the angular billiard 
holds:
$${\mathcal A}(A)=B.$$

\begin{figure}[h]
	\centering
	\includegraphics[width=0.9\linewidth]{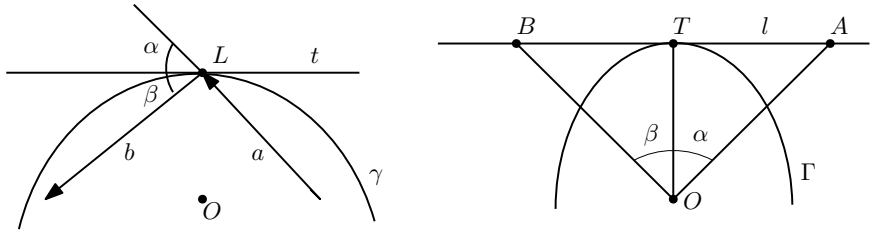}
	\caption{Polar duality, $\beta=\alpha$.}
	\label{01}
\end{figure}
The angular billiard is called {\it integrable} if there is a non-constant function
$G:U\backslash S\rightarrow {\mathbf R}$ which is invariant under the action of ${\mathcal A}$, i.e.
$$G(A)=G({\mathcal A}(A)), \qquad A\in U\backslash S.$$

\vskip03mm
\begin{example}
	Let $\Gamma$ be an ellipse defined by the equation
	$$
	\frac{x^2}{a^2}+\frac{y^2}{b^2}=1,\quad a>b>0,
	$$
	and $O(x_0,y_0)$ be an arbitrary point inside the ellipse. Then the angular billiard is integrable with the rational integral
	$$
	G(x,y)=\frac{b^2x^2+a^2y^2-a^2b^2}{(x-x_0)^2+(y-y_0)^2}.
	$$
\end{example}
\vskip03mm
It follows that if the Birkhoff billiard is algebraically integrable inside $\gamma$, then the angular billiard for the dual curve $\Gamma$ is integrable as well. Moreover, the integral of the angular billiard can be written explicitly in terms of the integral of Birkhoff billiard. More precisely, 
let $\Phi$ be a polynomial integral of Birkhoff billiard $\gamma$. One can assume (see \cite{bolotin}) that $\Phi$ is in the form
$$\Phi=\Phi(\sigma,v_1,v_2), \quad \sigma=xv_2-yv_1,$$
where $\Phi$ is a homogeneous polynomial in $\sigma, v_1,v_2$ of even degree. Moreover, one can assume that $\Phi$ vanishes on the tangent vectors to $\gamma$.

\begin{theorem}(\cite{BM1}) Let $\gamma$ be a closed convex curve and $\Phi(\sigma,v_1,v_2)$ be a homogeneous polynomial integral of even degree $n=2p$, vanishing on tangent vectors to $\gamma$.
Then the angular billiard for dual curve $\Gamma$ is also integrable with the integral of the form
$$
 G(x,y)=\frac{F(x,y)}{(x^2+y^2)^p},
$$
where the polynomial $F=\Phi(1,y,-x)$ is of degree $n$. Moreover, $F$  vanishes on $\Gamma$.
\end{theorem}

The angular billiard is an effective tool to study Birkhoff billiard.  Several new results on algebraic Birkhoff conjecture were obtained in \cite{BM1}, \cite{BM} using angular billiard. Let us describe these results. Denote by  $f$ the minimal defining polynomial of $\Gamma$. Since $F=0$ on $\Gamma$ we have:
$$
 F=f^k(x,y)g_1(x,y),
$$
where $g_1$ does not vanish identically on $\Gamma$. Let $\Gamma_1$ be an arc of $\Gamma$ where $g_1>0$. Then $\tilde{G}=G^{1/k}$ is also an integral of angular billiard and
$$
 \tilde G(x,y)=\frac{\tilde F(x,y)}{(x^2+y^2)^m},\quad \tilde F:=f(g_1)^{1/k}=fg,
$$ where $g:=g_1^{1/k},\   m:=p/k$.
The property that $\tilde G$ is an integral of angular billiard implies the remarkable identity:
\begin{equation}\label{eq1}
 \left(-\frac{\mu}{\varepsilon}\right)^{2m}\tilde F(x+\varepsilon \tilde F_y,y-\varepsilon \tilde F_x)=\tilde F(x+\mu \tilde F_y,y-\mu \tilde F_x)
\end{equation}
for small real $\varepsilon$, $(x,y)\in\Gamma_1$,
$$
  \mu=-\frac{(x^2+y^2)\varepsilon}{x^2+y^2+2\varepsilon(x\tilde F_y-y\tilde F_x)}.
$$
From (\ref{eq1}) it follows that on every tangent line to the completion of $\Gamma$ in ${\mathbf C}P^2$ acts a projective involution leaving invariant the set of intersection points of the tangent line with the projective curve. Equation (\ref{eq1}) yields  the following identity (Theorem 6.1 in \cite{BM1}).

\begin{theorem} (\cite{BM1}) The following formula holds true for all $(x,y)\in\Gamma_1$:
\begin{equation}\label{eq2}
g^3(x,y)H(f(x,y))=c_1(x^2+y^2)^{3m-3},
\end{equation}
where $c_1$ is a non-zero constant and
$$H(f):=
f_{xx}f_y^2-2f_{xy}f_xf_y+f_{yy}f_x^2
$$
is the affine Hessian of function $f$.
\end{theorem}

Identity (\ref{eq2}) implies the following results:

\begin{theorem}\label{th2.4} (\cite{BM1})
 Suppose that Birkhoff billiard inside
$\gamma$ admits a non-constant polynomial integral $\Phi$ on the
energy level $\{|v|=1\}.$ Let $\Gamma$ be the polar dual curve to $\gamma$, and $\tilde{\Gamma}$ be the corresponding irreducible curve in $\mathbf{C}P^2.$ Then, either
 $\tilde{\Gamma}$ has degree 2, or $\tilde{\Gamma}$ necessary contains singular points.
 Moreover, all singular and
inflection points of $\tilde{\Gamma}$ in $\mathbf{C}P^2$ belong to
the union of the isotropic lines defined by the equations
$$
L_+=\{ x+iy=0\}, \qquad L_-=\{x-iy=0\}.
$$
\end{theorem}

\vskip6mm \begin{corollary} \label{cor1}  If the
Birkhoff billiard inside $\gamma$ is integrable  with an integral which is polynomial in
$v$, then $\tilde{\gamma}$ does not have two real algebraic
ovals having a common tangent line (see Fig.\ref{0}).
\end{corollary}

Let us give the simplest example to Corollary \ref{cor1}. Consider the real algebraic curve
$$
 y^2=F(x)=(x-x_1)(x-x_2)(x-x_3)(x-x_4)f(x),\quad
  x_1<x_2<x_3<x_4,
$$
where $f(x)$ is a real polynomial such that $F(x)> 0$ for $x\in
(x_1,x_2)$ and $x\in (x_3,x_4).$
Then the Birkhoff billiard inside the real analytic oval
$$
 \gamma=\{(x,\pm\sqrt{F(x)}),\ x\in[x_1,x_2]\}
$$
does not admit polynomial integral, since the algebraic curve has
another analytic oval $\gamma_1$ as in Fig.\ref{0}.
$$
 \gamma_1=\{(x,\pm\sqrt{F(x)}),\ x\in[x_3,x_4]\}.
$$
\begin{figure}[h]
	\centering
	\includegraphics[width=0.5\linewidth]{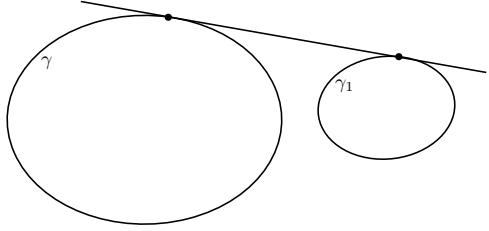}
	\caption{Non-integrable Birkhoff billiard inside $\gamma$.}
	\label{0}
\end{figure}

\vskip3mm \begin{corollary} \label{cor2}   Assume that
$\tilde{\Gamma}$ is a non-singular curve of degree $>2$ in
$\mathbf{C}P^2$ and has a  smooth real oval $\Gamma$ (for example,
$\tilde{\Gamma}$ is a nonsingular cubic). Then the dual curve
$\gamma$ is also an oval and Birkhoff billiard inside $\gamma$ is
not integrable by Theorem \ref{th2.4}.
\end{corollary}

Additional application of angular billiard yields the following result.

\vskip3mm

\begin{theorem} (\cite{BM})
  For any smooth closed convex curve $\gamma$ different from ellipse
Birkhoff billiard inside $\gamma$ does not admit polynomial integral
of degree
 4.
\end{theorem}
 \vskip3mm

We do not know how to prove this result directly without passing to angular billiard.

The crucial step in the proof of the algebraic Birkhoff conjecture was done by A. Glutsyuk \cite{g1}. The brief scheme of his proof is the following.

Every local branch of a germ of $\tilde{\Gamma}$ at a point $C\in L_{\pm}\cap \tilde{\Gamma}$
in adapted coordinates can be parametrized as follows
\begin{equation}\label{branch}
 t\rightarrow (t^q,ct^p(1+o(1)))
\end{equation}
as $t\rightarrow 0$, $q,p\in {\mathbf N}$,  $1\leq q<p, \ c\ne 0$. We refer to $\frac{p}{q}$ as {\it Puiseux} index.

The branch is called quadratic, if $\frac{p}{q}=2$, and sub-quadratic if $\frac{p}{q}<2$.
From the local behavior of the involution on the tangent lines to $\tilde\Gamma$ in the neighborhood of singular and inflection points  and from the identity (\ref{eq2}) one concludes:

\vskip3mm

\begin{theorem}\label{th2.8} (\cite{g1})
The curve $\tilde\Gamma$ satisfies the following conditions:

(i) Every branch transversal to $L_{\pm}$ at $0=L_+\cap L_-$ is quadratic.

(ii-a) Every branch tangent to $L_-$ or $L_+$ at $C\ne 0$ is quadratic.

(ii-b) Every branch transversal to $L_{\pm}$ at $C\ne 0$ is regular and quadratic.
\end{theorem}

The last step in the proof of the algebraic Birkhoff conjecture is the following general statement (see also \cite{g2}):

\begin{theorem}\label{th2.10} (\cite{g1})
Let $\gamma\in{\mathbf C}P^2$ be irreducible algebraic curve different from a line, such that
all its singular and inflection points belong to $L_{\pm}.$
If $\gamma$ satisfies the conditions (i), (ii-a), (ii-b) of Theorem \ref{th2.8}
then $\gamma$ is a conic.
\end{theorem}

The proof ot Theorem \ref{th2.10} is based on joint results of A. Glutsyuk and E.~Shustin \cite{g2}.

\section{\bf Billiards on constant curvature surfaces}

In this section we discuss algebraic Birkhoff conjecture for billiards on a surface $\Sigma$ of constant curvature $K=\pm 1$. Let $\Omega$ be a convex domain $\Omega\subset\Sigma$ with smooth boundary $\gamma=\partial\Omega$. Inside $\Omega$  particle moves along  geodesics. Reaching the boundary the  particle reflects according to the geometric optics low with respect to the Riemannian metric on $\Sigma$. We realize $\Sigma$ as the unit sphere in euclidean space ${\mathbf R}^3$  $$x_1^2+x_2^2+x_3^2=1,$$ for the case $K=1$, and as the upper sheet of the hyperboloid
$$x_1^2+x_2^2-x_3^2=-1$$ in ${\mathbf R}^3$ with the metric $ds^2=dx_1^2+dx_2^2-dx_3^2$, for $K=-1$.
We denote by $\hat{\gamma}$ the image of $\gamma$ under the projection ${\mathbf R}^3\backslash 0\rightarrow {\mathbf R}P^2$. There is an example of integrable billiard on $\Sigma$. Let $\Omega\subset\Sigma$ be the domain with the boundary $\gamma=\Sigma\cap\{ax_1^2+bx_2^2+cx_3^2=0\}$. Then Birkhoff billiard in $\Omega$ admits a first integral which is quadratic in velocities (\cite{bolotin2}, see also \cite{V}). We denote by $\Lambda\subset{\mathbf C}P^2$ the absolute, defined by the equation
$$
 \Lambda=\{(x_1:x_2:x_3):x_1^2+x_2^2\pm x_3^2=0\},
$$
where "+" is taken for $K=1$ and "-" for $K=-1$. S. Bolotin proved the following theorem.

\vskip3mm

\begin{theorem} (\cite{bolotin2})
Let $\Omega\subset\Sigma$ be a convex domain with a smooth boundary $\gamma$. Suppose that Birkhoff billiard inside $\Omega$ admits a non-constant polynomial integral on the energy level $\{|v|=1\}.$ It then follows that $\hat{\gamma}$ is necessarily an algebraic curve. Moreover, let $\tilde{\gamma}$ be corresponding to $\hat{\gamma}$ the irreducible curve in ${\mathbf C}P^2$. If $\tilde{\gamma}$ is a smooth curve, such that at least one intersection point of $\tilde{\gamma}$ with the absolute $\Lambda$ is transversal, then $\tilde{\gamma}$ is of degree 2.
\end{theorem}
\vskip3mm

The methods of the previous section are applicable to the Birkhoff billiard on $\Sigma$. In particular the following theorem holds.

\vskip3mm
\begin{theorem} (\cite{BM2})
Let $\tilde{\Gamma}$ in ${\mathbf C}P^2$ be the dual curve of $\tilde{\gamma}$. Then the following alternative holds: either $\tilde{\Gamma}$ is a conic, or $\tilde{\Gamma}$ necessarily contains
singular points, so that all the singular and all inflection points of $\tilde{\Gamma}$ belong to the absolute $\Lambda$.
\end{theorem}

The following theorems are analogous to Theorem \ref{th2.8} and Theorem \ref{th2.10}.
\vskip3mm

\begin{theorem}\label{th3.3} (\cite{g1})
	Curve $\tilde\Gamma$ has the following property:
For every intersection point $C$ of $\tilde{\Gamma}$ with $\Lambda$

(a) Every branch tangent to $\Lambda$ at $C$ is quadratic;

(b) Every branch transversal to $\Lambda$ at $C$ is regular and quadratic.
\end{theorem}

The algebraic Birkhoff conjecture for constant curvature surfaces follows from the following general fact:
\vskip5mm

\begin{theorem} (\cite{g1})
Let $\gamma\in{\mathbf C}P^2$ be irreducible algebraic curve different from a line such that
all its singular and inflection points belong to the  absolute $\Lambda$.
If $\gamma$ satisfies the conditions (a), (b) of Theorem \ref{th3.3},
then $\gamma$ is a conic.
\end{theorem}

\section{\bf Magnetic billiards}
Magnetic billiards is a very important and interesting class of billiards. There is an evidence (see i.e. \cite{BR}) that integrable magnetic billiards are very rare. In the sequel we describe several models of magnetic billiards and explain how the algebraic technique can be applied to test their integrability.

One considers the influence of a magnetic field of
constant magnitude $\beta>0$ on the billiard motion inside a convex domain
$\Omega\subset\mathbf{R}^2$ bounded by a simple smooth
closed curve $\gamma$. A
particle moves inside $\Omega$ with unit speed along a Larmor circle
of constant radius $r=\frac{1}{\beta}$ in counterclockwise
direction. Upon hitting the boundary, the particle is reflected
according to the law of geometric optics. This is the model of an ordinary
magnetic billiard.
On the other hand if the magnetic field has constant magnitude but changes sign after every collision with the boundary we call such a model two-sided magnetic billiard.
One can also study the case of a magnetic billiard on Sphere or Hyperbolic plane.
For the case of the surfaces of constant curvature $K=\pm 1$ and for constant magnetic field $\beta$ the motion between the collisions is along curves of constant geodesic curvature $\beta$. On the Sphere these are always Larmor circles of constant geodesic radius $r$, where  $\beta=\cot r.$ While on the Hyperbolic plane the curves of constant geodesic curvature $\beta$ are geodesic circles only when $\beta>1$. In this case the geodesic radius of the Larmor circles is given by $\beta=\coth r$. The case $\beta\leq 1$ is also very interesting, but will not be considered here.
\begin{figure}[h]
	\centering
	\includegraphics[width=0.5\linewidth]{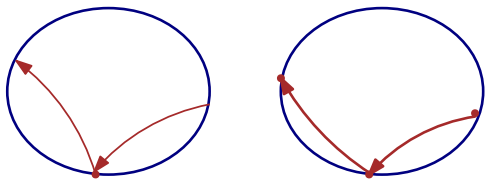}
	\caption{Ordinary vs. two-sided magnetic billiard.}
	\label{1}
\end{figure}
For all the models described above we shall assume that the boundary $\gamma$ of the domain $\Omega$ satisfies
$$
0<\beta< k_{\min}:=\min_{\gamma} k,
$$
where $k$ is the curvature of the boundary. In other words, we assume that the
magnetic field is relatively weak with respect to the curvature. In this
case billiard dynamics is correctly defined because  the boundary of the domain $\Omega$ is strictly convex with
respect to the circles of radius $r.$ So
the intersection of any circle of radius $r$ with
$\Omega$ consists of at most one arc. Moreover, under this
assumption, if a circle of radius $r$ oriented in the same direction
as the boundary is tangent to $\partial\Omega$ (with the agreed
orientation), then it contains the domain $\Omega$ inside.

Ordinary Magnetic
billiards were studied in many papers; see, e.g.,
\cite{Berg}, \cite{B}, \cite{GB},
\cite{BR}, \cite{T2}, \cite{BM3}.
Two-sided magnetic billiards were introduced  in \cite{k}.
Magnetic billiards on surfaces of constant curvature were studied in \cite{GB}, \cite{B}, \cite{B2} and recently in \cite{BM4} by the algebraic
approach.

The key ingredient of our approach for all magnetic models is to consider the "dual" object to the billiard, namely we define the domain $\Omega_r$ which consists of all centers of Larmor circles intersecting the original domain $\Omega$. This domain is a natural phase space of magnetic billiard and is diffeomorphic to an annulus bounded by two convex curves $\gamma_{\pm r}$. These curves have many names depending on the science they appear. They are called parallel curves, equidistants, fronts or offset curves.
It turns out that if a magnetic billiard admits an  integral which is polynomial in velocities, then
these curves, as well as $\gamma$ itself, are algebraic curves.  Moreover, the integrability imposes severe restriction on the singularities of these algebraic curves in $\mathbf C^2$. It is natural to expect that a more detailed analysis leads to complete classification of integrable magnetic billiards.

Now we turn to the formulation of our results on ordinary magnetic billiards on the plane \cite{BM3} and on surfaces of constant curvature \cite{BM4}.

\subsection{Ordinary Magnetic billiard}

Denote by $f_{\pm r}$ the minimal defining
polynomials of the irreducible component in $\mathbf{C}^2$
containing $\gamma_{\pm r}$, respectively. Since the curves
$\gamma_{\pm r}$ are real, $f_{\pm r}$ have real coefficients.
Notice, that it may happen that both $\gamma_{\pm r}$ belong to the
same component, so that $f_{+r}=f_{-r}.$ For instance, this is the
case for parallel curves to $\gamma$ when $\gamma$ is an ellipse.
In this case $f_{-r}=f_{+r}$ is
an irreducible polynomial of degree $8$.
\begin{theorem}\label{main1}
	Let $\Omega$ be a convex bounded domain with a smooth
	boundary $\gamma=\partial\Omega$ with
	the curvature strictly greater than $\beta$. Suppose that the
	magnetic billiard in $\Omega$ admits a non-constant polynomial integral $\Phi.$
	Then
	the affine curves $\{f_{\pm
		r}=0\}$ are smooth in $\mathbf{C}^2.$
\end{theorem}
\begin{corollary}\label{all1}
	For any non-circular domain $\Omega$ in the plane, the magnetic
	billiard inside $\Omega$ is not algebraically integrable for all but
	finitely many values of $\beta$.
\end{corollary}
In many cases one can get non-integrability for all values of $\beta$, with no exception.
For instance for ellipses the parallel curves appear to be always singular and hence we get:
\begin{corollary}Let $\Omega$ be the interior of the standard ellipse
	$$\partial\Omega=\left\{\frac{x^2}{a^2}+\frac{y^2}{b^2}=1\right\},\quad 0<b<a.$$
	Then for any magnitude of the magnetic field
	$0<\beta<k_{\min}=\frac{b}{a^2}$, the
	magnetic billiard in the ellipse is not algebraically integrable.
\end{corollary}
\subsection{Magnetic billiards on constant curvature surfaces}
Let $\Sigma$ be a surface of constant curvature $\pm 1$ realized in $\mathbf R^3$ as the unit sphere, for $K=1$, and as the upper sheet of the hyperboloid, for $K=-1$. 
Consider  a convex bounded domain $\Omega$ with smooth
boundary $\gamma=\partial\Omega$ lying on $\Sigma$ and the magnetic billiard inside $\Omega$.
In this case parallel curves $\gamma_{\pm r}$ on $\Sigma$ determine algebraic curves $\tilde\gamma_{\pm r}$ in $\mathbf C^3$ and we have the following:
\begin{theorem}\label{main2}
	Let $\Omega$ be a convex bounded domain with smooth
	boundary $\gamma=\partial\Omega$ which has geodesic curvature strictly greater than $\beta>0$ ($\beta>1$ in the Hyperbolic case). Suppose that the
	magnetic billiard in $\Omega$ admits a non-constant polynomial integral $\Phi.$
	Then if  $\gamma$  is not
	circular, then the curves $\tilde\gamma_{\pm r}$ are smooth algebraic curves in $\mathbf{C}^3$.
\end{theorem}

\begin{corollary}\label{all2}
	For any non-circular domain $\Omega$ on $\Sigma$, the magnetic
	billiard inside $\Omega$ is not algebraically integrable for all but
	finitely many values of $\beta$.
\end{corollary}
As an example we consider the magnetic billiard inside spherical ellipse and conclude the non-existence of polynomial integral for all
magnitudes of magnetic field:

\vskip03mm

\begin{example}

Let $\Omega$ be the interior of the ellipse on the sphere, i.e. the intersection of the sphere with a quadratic cone
$$\partial\Omega=\left\{\frac{x_1^2}{a^2}+\frac{x_2^2}{b^2}=x_3^2\right\},\quad 0<b<a.$$
The equation of parallel curves for the ellipse is defined by a polynomial $\hat{F}$ of degree eight (see Appendix of \cite{BM4}). The curve on the sphere $\{\hat{F}=0\}$ is singular for arbitrary $a$ and $b$ which imply algebraic non-integrability by Theorem \ref{main2}.
\end{example}

\section{\bf Polynomial integrals for two-sided magnetic billiards}
Let $\Omega$ be a convex domain in the plane bounded by a smooth simple closed  curve $\gamma$.
The magnetic field $\pm\beta$ is assumed to be relatively week $0<\beta<k_{\min}$. In such a case every Larmor circle is transversal to the boundary. We shall assume that after each collision the sign of magnetic field changes to the opposite, so that the Larmor circles
change their orientation from positive to negative (see Fig. \ref{1}).

It is a good idea to think of two copies of the domain $\Omega_1, \Omega _2$ glued along the boundary $\gamma$, and  the magnetic field is $+\beta$ on $\Omega_1$ and is $-\beta$ on $\Omega_2$. Thus altogether we get  a sphere with the magnetic field orthogonal to the surface of the sphere. We shall use subindex 1, 2 for points, indicating the sides of the domain $\Omega$.
We turn now to the definition of polynomial integral of motion of two-sided magnetic billiard.
\begin{definition}\label{def} Let $\Phi_1,\Phi_2:T_1\Omega\rightarrow\mathbf{R}$
	be two functions on the unit tangent bundle which are polynomial in the components of the unit tangent vector $v=(v_1,v_2)$
	$$\Phi_1(x,v)=\sum_{k+l=0}^N
	a^{(1)}_{kl}(x)v_{1}^kv_{2}^l, \quad \Phi_2(x,v)=\sum_{k+l=0}^N
	a^{(2)}_{kl}(x)v_{1}^kv_{2}^l $$
	with coefficients
	continuous up to the boundary. We
	call the pair of functions $\Phi_1,\Phi_2$ a polynomial integral of two-sided magnetic billiard if the
	following conditions hold.
	
	1. $\Phi_1$ are $\Phi_2$ keep constant values on the unit tangent vectors of positive and negative Larmor circles respectively.
	
	2. For the collision point $z\in \gamma$ and any vector $v\in T_z\Omega, |v|=1$
	$$\Phi_1(z,v)=\Phi_2(z,v-2\langle n,v\rangle n),$$
	where $n$ is the inward unit normal to $\partial\Omega$ at $z$.
\end{definition}
It is very clear that two-sided round disc provides an integrable example of the two-sided magnetic billiard:

\begin{example}\label{example}Let $\gamma$ be a unit circle centered at the
	origin, so $n(x)=-x$. Consider two functions
	$$\Phi_1(x,v)=x_1^2+x_2^2+\frac{2}{\beta}(v_1x_2-v_2x_1)=x^2-\frac{2}{\beta}<n,Jv>,$$
	$$\Phi_2(x,v)=x_1^2+x_2^2-\frac{2}{\beta}(v_1x_2-v_2x_1)=x^2+\frac{2}{\beta}<n,Jv>.$$
	It can be easily seen that they satisfy the conditions 1 and 2, thus providing the integral of two-sided billiard in the sense of Definition~\ref{def}.

\end{example}
Our main result for two-sided billiards reads as follows:
\begin{theorem}\label{main}
	Suppose the two-sided magnetic billiard admits a non-constant polynomial integral of motion.
	Then for any point $Q$, all the Puiseux indices (see (\ref {branch})) of all local branches at $Q$ of the curves $\gamma_{\pm r}$ are greater or equal then 2, i.e. the branches cannot be sub-quadratic.
\end{theorem}
\begin{remark}
	It is plausible that one can improve the method of proof in order to exclude all singularities. But we couldn't complete the
	analysis.
\end{remark}

We have the immediate corollaries:
\begin{corollary}\label{all}
	For any non-circular domain $\Omega$ in the plane, the two-sided magnetic
	billiard inside $\Omega$ is not algebraically integrable for all but
	finitely many values of $\beta$.
\end{corollary}
\begin{proof}
	Indeed, let us choose $r\in (r_{\min}, r_{\max})$, $r=1/\beta$, to be a regular value  the curvature radius function.
	For such $r$ (Fig. \ref{4}) the curve $\gamma_{+r}$ contains  $A_2$-cusps (see \cite{bruce}, Section 7.12). Since $A_2$ is stable singularity, then there is an open interval of $r$ where the parallel curves $\gamma_{+r}$ contain $A_2$-cusps. Therefore, for all but finitely many values of $\beta$ the curves $\gamma_{+r}$  contain sub-quadratic singularities. But, by Theorem \ref{main}, this can not happen for those $\beta$ when the billiard is integrable. This proves the corollary.
	
\end{proof}

\begin{figure}[h]
	\centering
	\includegraphics[width=0.5\linewidth]{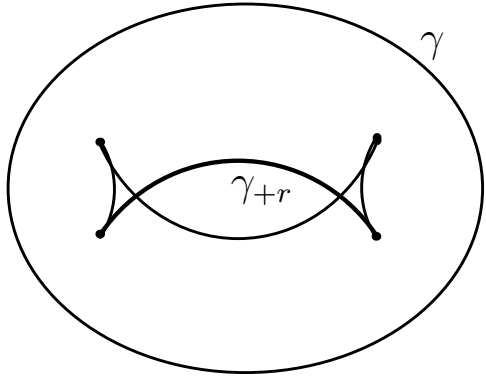}
	\caption{Cusps of the parallel curve.}
	\label{4}
\end{figure}
\vskip05mm

\begin{example} Let $\Omega$ be the interior of the ellipse
	$$
	\frac{x^2}{a^2}+\frac{y^2}{b^2}=1,\quad a>b>0.
	$$
	The equation of the parallel curves for ellipse reads:
	$$
	a^8 (b^4+(r^2-y^2)^2-2 b^2 (r^2+y^2))+b^4 (r^2-x^2)^2 (b^4-2 b^2 (r^2-x^2+y^2)+(x^2+y^2-r^2)^2)
	$$
	$$
	-2
	a^6 (b^6+(r^2-y^2)^2 (r^2+x^2-y^2)-b^4 (r^2-2 x^2+3 y^2)-b^2 (r^4+3 y^2 (x^2-y^2)+
	$$
	$$
	r^2 (3
	x^2+2 y^2)))+
	2 a^2 b^2 (-b^6 (r^2+x^2)-(-r^2+x^2+y^2)^2 (r^4-x^2 y^2-r^2 (x^2+y^2))+
	$$
	$$
	b^4 (r^4-3 x^4+3
	x^2 y^2+r^2 (2 x^2+3 y^2))+b^2 (r^6-2 x^6+x^4 y^2-3 x^2 y^4+r^4(-4 x^2+2 y^2)+
	$$
	$$
	r^2 (5 x^4-3 x^2 y^2-3 y^4)))+
	a^4 (b^8+2 b^6 (r^2+3 x^2-2 y^2)+(r^2-y^2)^2 (-r^2+x^2+y^2)^2-
	$$
	$$
	2 b^4 (3 r^4-3 x^4+5 x^2 y^2-3 y^4+4
	r^2 (x^2+y^2))+2 b^2 (r^6-3 x^4 y^2+x^2 y^4-2 y^6+
	$$
	$$
	2 r^4 (x^2-2 y^2)+r^2 (-3 x^4-3 x^2 y^2+5 y^4)))=0.
	$$
	This curve is irreducible and has $A_2$-cuspidal branches for every $r>0$. For arbitrary $a,b$ the formulas for the singular points are very complicated. We consider, for simplicity, the case $a=2,b=1$. Then the curve has the form
	$$
	(9 r^8-6 r^6 (15+2 x^2+7 y^2)+(x^2+4 y^2-4)^2 (x^4+2 x^2 (y^2-3)+(3+y^2)^2)+
	$$
	$$r^4
	(297-2 x^4+270 y^2+73 y^4+x^2 (62 y^2-90))+
	$$
	$$2 r^2 (2 x^6-x^4 (31+15 y^2)+x^2 (135+70 y^2-45 y^4)-4 (45+45 y^2+31 y^4+7 y^6)))=0.
	$$
	This curve has a singular point $(x_0,y_0)$, where
	$$
	x_0=\sqrt{\frac{16}{3}-4\ 2^{2/3} r^{2/3}+2\ 2^{1/3} r^{4/3}-\frac{r^2}{3}},
	$$
	$$
	y_0=\sqrt{2^{2/3} r^{2/3}-\frac{1}{3}-2\ 2^{1/3} r^{4/3}+\frac{4 r^2}{3}}.
	$$
	By direct calculation one can check that the curve has two $A_2$-cuspidal branches at this point of the form
	$$
	x=x_0+t^2,\qquad y=y_0+b_1t^2+b_2t^3+\dots
	$$
	Hence, for any magnitude of the magnetic field
	$0<\beta<k_{\min}$, the two-sided
	magnetic billiard in the ellipse is not algebraically integrable.
\end{example}

The rest of this Section is devoted to the proof of Theorem \ref{main}.

Larmor circles on $\Omega_1$ and $\Omega_2$ differ by their orientation. Denote $\Omega_r$ the set of all Larmor centers. This set is also has two sides $\Omega_{1 r},$  $\Omega_{2 r}$. The boundaries of these annuli consist of two curves $\gamma_{+r}, \gamma_{-r}$
which are parallel curves to $\gamma$.

For two-sided magnetic billiard the billiard map $\M$ acts from one side of $\Omega _r$ to the other, by the rule: the center of a Larmor circle is mapped to the center of the reflected Larmor circle, which belongs to the other side, because the sign of magnetic field changes.

Let us elaborate how the mapping $\M$ acts. Let $z$ be a point on $\gamma$ and the circle  $C_1$ on $\Omega_1$  incomes to $z$ and reflects to $C_2$ on the other side $\Omega_2$. Let $$z=\gamma(s), \quad z_+= z+rJ\dot \gamma(s)\in\gamma_{+r},\ z_-= z-rJ\dot \gamma(s)\in\gamma_{-r} .$$
Then for the center $P_1$  of the circle $C_1$ we have:
$$
P_1=z_+-rJ\dot \gamma(s)+rR_{-\epsilon}J\dot \gamma(s),
$$
where $R_{\epsilon}$ is counterclockwise rotation by $\epsilon$, and $\epsilon$ is the angle between $C_1$ and $\dot\gamma$ at $z$.
Then the point $Q_2:=\M(P_1)$ is obtained by the formula:
$$
Q_2=z_+-rJ\dot \gamma(s)-rR_{\epsilon}J\dot \gamma(s).
$$
Where indices 1, 2 for the points $P,Q$ indicates their side (see Fig. \ref{2}).

One can see that two-sided magnetic billiard map is symplectic with respect to the standard symplectic form (area form).
Remarkably,
\begin{equation}\label{M}
\M(P_1)=Q_2 \Rightarrow\M(Q_1)=P_2.
\end{equation}

Consider now  integrable two-sided magnetic billiard having pair of integrals $\Phi_1, \Phi_2$. 
We introduce the mapping $\mathcal {L}_i:\Omega_i\rightarrow\Omega_{i r}$ which assigns to $(x,v)$ the center of the  Larmor circle passing through $(x,v)$ with a given orientation.
Moreover, for the function $\Phi_i$
we define the function $F_i:\Omega_{i r} \rightarrow \mathbf{R}$ by the rule:
$$ F_i\circ\mathcal L_i=\Phi_i.$$

It then follows from Theorem 1.2 of \cite{BM3} that the functions $F_i$ are, in fact, polynomials.
Moreover, since $\Phi_i$ are integrals of the two-sided billiard, we have
\begin{equation}\label{F}
F_1(P_1)=F_2(Q_2), \quad F_1(Q_1)=F_2(P_2)
\end{equation}
\begin{remark}
	It follows from (\ref{F}) that we may assume without loss of generality that $F_1=F_2=:F$. Since one can pass to
	$(F_1+F_2)$ or to $(F_1-F_2)^2$ .
\end{remark}
So we get from (\ref{F}) a remarkable equations on the polynomial functions $F$:
\begin{equation}\label{remarkable}
F(z_+-rJ\dot \gamma(s)+rR_{-\epsilon}J\dot \gamma(s))=F(z_+-rJ\dot \gamma(s)-rR_{\epsilon}J\dot \gamma(s));
\end{equation}

\begin{figure}
	\centering
	\includegraphics[width=0.5\linewidth]{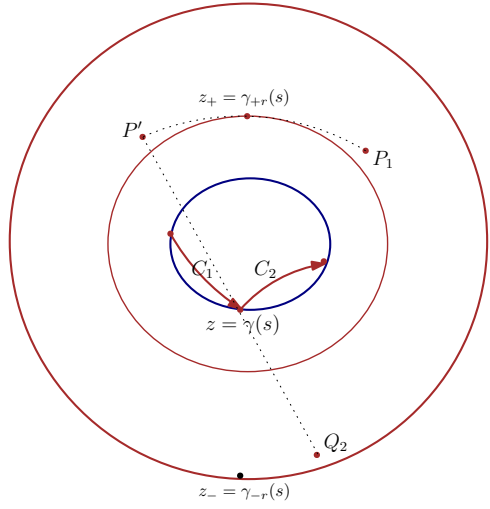}
	\caption{The point $P$, center of the circle $C_1$  is mapped to $Q$, the center of $C_2$.}
	\label{2}
\end{figure}

In particular for $\epsilon=0$ we get
\begin{equation}\label{x}
F(z_+)=F(z_-).
\end{equation}
\begin{proposition}\label{prop}
	Polynomial $F$ keeps constant values on both boundaries:
	$$F|_{\gamma_{\pm r}}=c.$$
\end{proposition}
\begin{remark}\label{r}
	Replacing $F$ with $(F-c)$ we shall assume from now on that $F$ vanishes on both boundaries.
\end{remark}
\begin{corollary}
	The curves $\gamma_{\pm r}$, and hence also $\gamma$ are algebraic curves.
\end{corollary}
\begin{proof}[Proof of Proposition \ref{prop}]
	Differentiating (\ref{remarkable}) with respect to $\epsilon$	at $\epsilon=0$ we get
	\begin{equation}\label{DF1}
	DF|_{z_+}(r\dot\gamma)=DF|_{z_-}(r\dot\gamma).
	\end{equation}
	In addition, differentiating (\ref{remarkable}) with respect to $s$ at $\epsilon=0$ we get, using Frenet formulas
	\begin{equation}\label{DF2}
	DF|_{z_+}((1-k(s)r)\dot\gamma)=DF|_{z_-}((1+k(s)r)\dot\gamma).
	\end{equation}
	Since the determinant of the matrix
	$$
	\begin{pmatrix}
	r & r   \\
	(1-kr) & (1+kr)
	\end{pmatrix}$$
	does not vanish, we conclude that the linear equations (\ref{DF1}), (\ref{DF2}) have only trivial solution:
	$$DF|_{z_+}(\dot\gamma)=DF|_{z_-}(\dot\gamma)=0.$$
	Since the vectors $\dot\gamma(s)$ and $\dot\gamma_{\pm r}(s)$  are proportional,  then $F$ must be constant on each boundary curve. Moreover, the corresponding constants for both boundaries must be equal by (\ref{x}).
\end{proof}

Let us denote by $f_+,\ f_-$ minimal defining polynomials of algebraic curves $\gamma_{\pm r}$ respectively. There are two possibilities:

Case 1. $f_+=f_-$, i.e. the curves $\gamma_{\pm r}$ belong to the same component.
In this case we have $$F=(f_+)^k\ g_1.
$$

Case 2. $f_+\neq f_-$, i.e.   the curves $\gamma_{\pm r}$ belong to different components.
Then
$$
F=(f_+)^k(f_-)^l\ g_1.
$$
In both cases polynomial $g_1$ does not vanish on $\gamma_{\pm r}$ except  finitely many points.
Moreover in the second case we claim:
\begin{proposition}\label{kl}
	In Case 2  the numbers $k$ and $l$ are equal.
\end{proposition}
\begin{proof}
	If on the contrary $l>k$, we introduce the function $$\tilde F=F^\frac{1}{k},$$ which is at least $C^1$-smooth and satisfies (\ref{F}) together with $F$. Take the points $P,Q,\M (P)=Q$, and consider the points $X$ and $Y$ on the curves $\gamma_{+r}, \gamma_{-r}$ close to $P,Q$ as it is shown on the Fig. \ref{3}.
	Then the distances  $|P-X| $ and $|Q-Y|$ are of order $\epsilon^2$ and we have by mean value theorem for $\tilde F$ :
	$$
	\lim\limits_{\epsilon\rightarrow 0}\frac{|\tilde F(P)-\tilde F(X)|}{|P-X|}=|\nabla\tilde F(z_+)|,\  \lim\limits_{\epsilon\rightarrow 0}\frac{|\tilde F(Q)-\tilde F(Y)|}{|Q-Y|}=|\nabla\tilde F(z_-)|.
	$$
	Moreover, it is easy to compute
	$$
	\lim\limits_{\epsilon\rightarrow 0}\frac{|P-X|}{|Q-Y|}=\frac{\rho_-}{\rho_+}.
	$$
	Here $$\rho_+=\frac{1}{k_+}=r-\rho, \rho_-=\frac{1}{k_-}=r+\rho$$ are
	the curvature radii of $\gamma_{+r},\gamma_{-r}$ expressed via curvature radius of $\gamma$ (see below).
	Hence using (\ref{F}) for $\tilde F$  we get :
	\begin{equation}\label{nabla}
	\frac{|\nabla\tilde F(z_+)|}{|\nabla\tilde F(z_-)|}=\frac{\rho_+}{\rho_-}.
	\end{equation}
	Since $r>\rho$, then $\rho_+, \rho_-$ are finite and different from 0, then the gradients of $\tilde F$ at the points $z_{\pm}$ vanish or not
	simultaneously. But on the other hand we have:
	$$
	\tilde F=F^{\frac{1}{k}}=f_+f_-^{\frac{l}{k}}g^{\frac{1}{k}}.
	$$
	Thus $|\nabla\tilde F(z_+)|$ does not vanish but $|\nabla\tilde F(z_-)|$ does, since  $l/k>1$. Contradiction.
\end{proof}
\begin{remark}
	In Case 1 formula (\ref{nabla}) also holds true for the function $\tilde F=F^\frac{1}{k}$  and the proof is verbatim.
\end{remark}

\subsection {Differential Geometric computations}Denote by $\rho, \rho_+, \rho_-$ the curvature radii of the curves $\gamma$, $\gamma_{+r},\gamma_{-r}$ respectively.
It then follows from the properties of parallel curves that curvature radii are related at the corresponding points:
$$\rho_+=\frac{1}{k_+}=r-\rho,\  \rho_-=\frac{1}{k_-}=r+\rho.$$
It will be important that

$$
0<\rho_+<r<\rho_-.
$$

We choose $s$ to be the arc length parameter of $\gamma$ according with the counter clockwise orientation.
Then we have
\begin{equation}\label{dot}
\dot\gamma_{+r}=\frac{d}{ds}(\gamma+rJ\dot\gamma)=(1-kr)\dot\gamma=-\frac{\rho_+}{\rho}\dot\gamma,
\end{equation}
$$
\dot\gamma_{-r}=\frac{d}{ds}(\gamma-rJ\dot\gamma)=(1+kr)\dot\gamma=\frac{\rho_-}{\rho}\dot\gamma.
$$

Next we recall the formulas for the curvature $k$ of the non-singular curve defined
implicitly by $\{f=0\}$ with respect to the normal $\nu=(f_x,f_y)$:
\begin{equation}\label{k}
k=-\frac{H(f)}{|\nabla f|^3}.
\end{equation}
Indeed,
positive tangent vector corresponding to $\nu$ is $v=\dot\gamma=(f_y,-f_x)$. Differentiate with respect to $s$ the identity $<\nabla f,\dot\gamma>=0 $ we get:
$$
<\nabla f,\ddot\gamma>+<D^2f\dot\gamma,\dot\gamma>=0,
$$
which together with Frenet formulas gives
exactly (\ref{k}).
\subsection{Terms of order $\epsilon^3$ in the remarkable equation}

Denote by $$n(s)=J\dot\gamma(s).$$
Equation (\ref{remarkable})  can be written in the following form:
\begin{equation}\label{remarkable1}
F(z_+-r(I-R_{-\epsilon})n)=F(z_-+r(I-R_{\epsilon})n);
\end{equation}
\begin{figure}[h]
	\centering
	\includegraphics[width=0.5\linewidth]{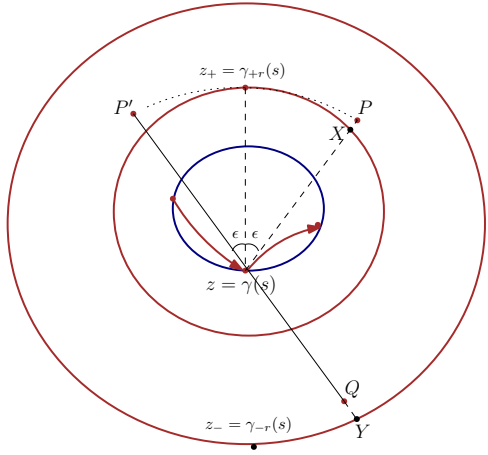}
	\caption{Point $P$ is mapped to $Q$.}
	\label{3}
\end{figure}

In the sequel we write equation (\ref{remarkable1}) for $$\tilde F=F^{\frac{1}{k}}=f\cdot g^\frac{1}{k}$$ in both Cases 1, 2. Where we write  $f:=f_+$ in Case 1 and $f:=f_+\cdot f_-$ in Case 2. Also we use the fact that   $\nabla f$ does not vanish on $\gamma_{+r},\gamma_{-r}$ except for finitely many points. Moreover, it is important that the normal $n(s)$ is given by: $$n(s)=\frac{\nabla \tilde F}{|\nabla \tilde F|}(\pm z)$$ for both points  $z_+=\gamma_{+r}(s)$ and $z_-=\gamma_{-r}(s)$. Using this we can rewrite (\ref{remarkable1}):

\begin{equation}\label{rem1} \tilde F\left(
x_+-r\frac{\tilde F_x(1-\cos\epsilon)-\tilde F_y\sin\epsilon}{|\nabla \tilde F|},
y_+-r\frac{\tilde F_y(1-\cos\epsilon)+\tilde F_x\sin\epsilon}{|\nabla \tilde F|} \right)=
\end{equation}
$$=\tilde F\left(
x_-+r\frac{\tilde F_x(1-\cos\epsilon)+\tilde F_y\sin\epsilon}{|\nabla \tilde F|},
y_-+r\frac{\tilde F_y(1-\cos\epsilon)-\tilde F_x\sin\epsilon}{|\nabla \tilde F|} \right).
$$
Notice that LHS is evaluated close to the point $z_+=(x_+,y_+)=\gamma_{+r}(s)$ while RHS close to the point $z_-=(x_-,y_-)=\gamma_{-r}(s)$.

The
next step is to expand equation (\ref{rem1}) in power series in
$\epsilon.$ Equating the terms of order $\epsilon^3$ on both sides we get:
\begin{equation}\label{rem2}
\frac{r^3}{6|\nabla \tilde F|^3}(\tilde F_{xxx}\tilde F_y^3-3\tilde F_{xxy}\tilde F_y^2\tilde F_x+3\tilde F_{xyy}\tilde F_y\tilde F_x^2-\tilde F_{yyy}\tilde F_x^3)\,-
\end{equation}
$$
\frac{r^2}{2|\nabla \tilde F|^2}(\tilde F_{xx}\tilde F_x\tilde F_y+\tilde F_{xy}(\tilde F_y^2-\tilde F_x^2)-\tilde F_{yy}\tilde F_x\tilde F_y)=
$$
$$=\frac{r^3}{6|\nabla \tilde F|^3}(\tilde F_{xxx}\tilde F_y^3-3\tilde F_{xxy}\tilde F_y^2\tilde F_x+3\tilde F_{xyy}\tilde F_y\tilde F_x^2-\tilde F_{yyy}\tilde F_x^3)\,+$$
$$
\frac{r^2}{2|\nabla \tilde F|^2}(\tilde F_{xx}\tilde F_x\tilde F_y+\tilde F_{xy}(\tilde F_y^2-\tilde F_x^2)-\tilde F_{yy}\tilde F_x\tilde F_y).
$$
Where LHS is computed at the point $z_+=(x_+,y_+)=\gamma_{+r}(s)$ while RHS-at the point $z_-=(x_-,y_-)=\gamma_{-r}(s)$.

Introduce the vector field $v=(\tilde F_y,-\tilde F_x)$. Equation (\ref{rem2}) can be written as follows:
\begin{equation}\label{rem3}
\frac{1}{|\nabla \tilde F(\gamma_{+r}(s))|^3}L_{v(\gamma_{+r}(s))}\left(H(\tilde F)-\beta|\nabla  \tilde F|^3\right)=
\end{equation}
$$
\frac{1}{|\nabla \tilde F(\gamma_{-r}(s))|^3}L_{v(\gamma_{-r}(s))}\left(H(\tilde F)+\beta|\nabla\tilde F|^3\right),
$$
where we used $H(\tilde F)$ for the affine Hessian.

It is convenient to pass to differentiation with respect to $s$.
We compute using (\ref{dot}):
$$
v(\gamma_{+r}(s))=-|\nabla \tilde F(\gamma_{+r}(s))|\dot\gamma(s)=-|\nabla \tilde F(\gamma_{+r}(s))|\frac{\rho}{\rho_+}\dot\gamma_{+r}(s),
$$
$$
v(\gamma_{-r}(s))=|\nabla \tilde F(\gamma_{-r}(s))|\dot\gamma(s)=|\nabla \tilde F(\gamma_{-r}(s))|\frac{\rho}{\rho_-}\dot\gamma_{-r}(s).
$$
Then substituting these formulas into (\ref{rem3}) we get:
\begin{equation}\label{rem4}
-\frac{\rho}{\rho_+|\nabla \tilde F(\gamma_{+r}(s))|^2}\frac{d}{ds}|_{\gamma_{+r}(s)}\left[\left(H(\tilde F)-\beta|\nabla \tilde F|^3\right)(\gamma_{+r}(s))\right]=
\end{equation}
$$
\frac{\rho}{\rho_-|\nabla \tilde F(\gamma_{-r}(s))|^2}\frac{d}{ds}|_{\gamma_{-r}(s)}\left[\left(H(\tilde F)+\beta|\nabla \tilde F|^3\right)(\gamma_{-r}(s))\right].
$$
Using (\ref{k}) we get from (\ref{rem4}):
\begin{equation}\label{rem5}
-\frac{\rho}{\rho_+|\nabla \tilde F(\gamma_{+r}(s))|^2}\frac{d}{ds}|_{\gamma_{+r}(s)}\left[|\nabla \tilde F|^3(\gamma_{+r}(s))\left(k_+(s)-\beta\right)\right]=
\end{equation}
$$
\frac{\rho}{\rho_-|\nabla \tilde F(\gamma_{-r}(s))|^2}\frac{d}{ds}|_{\gamma_{-r}(s)}\left[|\nabla \tilde F|^3(\gamma_{-r}(s))\left(-k_-(s)+\beta\right)\right].
$$
Let us introduce the function $u(s):=|\nabla \tilde F(\gamma_{+r}(s))|^3 $. Then by (\ref{nabla}) (see Remark 3) we have $$|\nabla \tilde F(\gamma_{-r}(s))|^3 =
u(s)\frac{\rho_-^3(s)}{\rho_+^3(s)}.$$
Substituting this into (\ref{rem5}) we get a linear differential equation on $u$:
\begin{equation}\label{linear}
\frac{d}{ds}\left(u(s)\left(\frac{1}{\rho_+(s)}-\frac{1}{r}\right)\right)-\frac{\rho_+^3(s)}{\rho_-^3(s)}\frac{d}{ds}\left(u(s)\frac{\rho_-^3(s)}{\rho_+^3(s)}\left(\frac{1}{\rho_-}-\frac{1}{r}\right)\right)=0
\end{equation}
Equation (\ref{linear}) can be written in the form:
\begin{equation}\label{linear1}
A(s)u'(s)+B(s)u(s)=0,
\end{equation}
where $A,B$ can be computed:
$$
A(s)=\left(\frac{1}{\rho_+(s)}-\frac{1}{r}\right)-\left(\frac{1}{\rho_-(s)}-\frac{1}{r}\right)=\frac{2\rho}{r^2-\rho^2},
$$
$$
B(s)=A'(s)-3\left(\frac{\rho_+}{\rho_-}\right)\left(\frac{1}{r+\rho}-\frac{1}{r}\right)\left(\frac{r+\rho}{r-\rho}\right)'=
$$
$$
=A'(s)+6\frac{\rho\rho'(s)}{(r+\rho)^2(r-\rho)}.
$$
In order to integrate (\ref{linear1}) we compute the integrating factor
$$
\mu=e^{\int \frac{B(s)}{A(s)}\ ds}=A(s)e^{3\int \frac{d\rho}{(r+\rho)}\ }=\frac{2\rho(r+\rho)^2}{r-\rho} .
$$
With the help of the integrating factor we get the solution:
$$
\mu\cdot u(s)=C.
$$
Here $C$ is a constant different from $0$, since otherwise $u$ is zero identically but $u$ does not vanish by the definition.
Substituting the expression for $\mu$  via $\rho_+$ we get:
$$
2 u\cdot\frac{(2r-\rho_+)^2(r-\rho_+)}{\rho_+}=C.
$$

Using formula (\ref{k}) for the curvature we get the identity which is valid for any point of the curve $\gamma_{+r}$:
\begin{equation}
2\frac{|\nabla \tilde F|^3\left(2r-\frac{|\nabla \tilde F|^3}{H(\tilde F)}\right)^2\left(r-\frac{|\nabla \tilde F|^3}{H(\tilde F)}\right)}{\frac{|\nabla \tilde F|^3}{H(\tilde F)}}=C.
\end{equation}
So we have:
\begin{equation}\label{final}
2\left(2rH(\tilde F)-|\nabla \tilde F|^3\right)^2\left(rH(\tilde F)-|\nabla \tilde F|^3\right)=C\cdot H(\tilde F)^2.
\end{equation}

In order to analyze this equation, consider a singular point $Q$ of algebraic curve $\{f_+=0\}$.  Let $t$ be a local parameter  along a local branch of the curve at $Q$ given by the equations: $$
y=c\ t^p(1+o(1)),
x=t^q;\ p,q \in \mathbf N,\  p>q\geq 1.
$$In order to prove Theorem \ref{main} we need to show that $\frac{p}{q}\geq 2$.
We can write $$H(\tilde F)\sim t^a\quad {\rm and}\quad  |\nabla \tilde F|^3\sim t^b,\quad a,b>0$$  near $Q$ along the branch, where we compute $a,b$ via $p,q$ below.
If $b\geq a$ then we have that  the LHS of (\ref{final}) is of order  $t^{3a}$, while the RHS is of order $t^{2a}$, impossible.
Thus we must have $a>b>0$, and then the LHS is of order $t^{3b}$ and we get:
\begin{equation}\label{32}
3b=2a.
\end{equation}
It then follows from (\ref{32}) that 
\begin{equation}\label{a-b}
a-b=\frac{a}{3}>0.
\end{equation}

Next we have for the branch:
$$y-\phi(x)=0,\quad
\phi=c\ x^{\frac{p}{q}}+o(x^{\frac{p}{q}}).
$$
Moreover, we can factorize (see for instance Theorem 2.6.6 of \cite{Wall}): $$\tilde F=(y-\phi(x)) g,$$ then we have along the branch (more precisely, along a pro-branch defined over a sector, see Sections 2.1 and 4.1 of \cite{Wall}):
$$
H(\tilde F)=g^3 H(y-\phi(x))=-g^3\cdot \phi''(x),
$$
$$
|\nabla \tilde F|^3=g^3|\nabla(y-\phi(x))|^3=g^3(1+\phi'(x)^2)^{\frac{3}{2}}.
$$
So we compute:
$$
a=3d+\left(\frac{p}{q}-2\right)q,\quad
b=3d,
$$where $d$ is the order of the function $g$.
Therefore we have:
$$
a-b=\left(\frac{p}{q}-2\right)q.
$$
This, and  (\ref{a-b}) imply that $
\frac{p}{q}>2
$,
completing the proof of Theorem \ref*{main}.
\section*{\bf Acknowledgments}We are grateful to Alexey Glutsyuk,
Eugene Shustin and Uriel Sinichkin for very helpful discussions.

\end{document}